\documentclass[a4paper,twoside,12pt]{amsart}
               
\usepackage{latexsym}
\usepackage{amsmath}
\usepackage{multirow}
\usepackage{epsfig}
\usepackage{amsfonts}
\usepackage{setspace} 
\usepackage{t1enc} 
\usepackage{capt-of}
\usepackage{natbib}
\usepackage{booktabs}  
\usepackage{hyperref}
\usepackage{array}
\usepackage{leftidx}
\usepackage{psfrag}

\usepackage{marginnote} 
   
\newcommand\mn[1]{
}

\newcommand\omn[1]{
}

\newtheorem{Theorem}{Theorem} 
\newtheorem{Proposition}{Proposition}

\theoremstyle{definition}

\theoremstyle{remark} 
\newtheorem{Remark}{Remark}

 \newcommand{\mcond} 
{\mbox{\hspace{.4ex}\rule[-0.3ex]{.2ex}{2ex}\hspace{.45ex}}} 

\begin{document}

\title{Signals Featuring Harmonics with Random Frequencies - Spectral, Distributional and  Ergodic  Properties}
\author[A. Baxevani]{Anastassia Baxevani}
\address{Department of Mathematics and Statistics\\
University of Cyprus\\
PO Box 20537\\
1678, Nicosia, Cyprus }
 	\email{baxevani@ucy.ac.cy}
\author[K. Podg\'orski]{Krzysztof Podg\'orski}
\address{Centre for Mathematical Sciences\\
Mathematical Statistics\\
Lund University\\
Box 118\\
SE-22100 Lund, Sweden}
\email{krys@maths.lth.se}

 	\date{\today}
\maketitle

\begin{abstract}
It has been observed that an interesting class of non-Gaussian stationary processes is obtained when  in the harmonics of a signal with random amplitudes and phases, frequencies can  also vary randomly. In the resulting models, the statistical distribution of  frequencies determines the process spectrum while the distribution of amplitudes governs the process distributional properties. Since  decoupling the distributional and spectral properties can be advantageous in applications, we thoroughly investigate a variety of properties exhibited by these models. 

We extend  previous work  that represented processes  as finite sum of harmonics, by conveniently embedding them into the class of harmonizable processes. Harmonics are integrated with respect to  independently scattered second order non-Gaussian random measures. The proposed approach provides with a proper mathematical framework that allows to study spectral, distributional, and ergodic properties. The mathematical elegance of these  representations avoids serious conceptual and technical difficulties with limiting behavior of the models while at the same time facilitates derivation of their fundamental properties. 

In particular, the  multivariate distributions are obtained  and   the asymptotic behavior of time averages is formally derived through the strong ergodic theorem.  Several deficiencies following from the previous approaches are resolved and some of the results  appearing in the literature are corrected and extended.  It is shown that due to the lack of ergodicity processes exhibit an interesting property of non-trivial randomness remaining in the limit of time averages. This feature maybe utilized to modelling signals observed in the presence of influential and variable random factors. 

The case of a stationary process with double exponential (Laplace) distribution that is important in many signal processing applications is discussed in full detail. Finally explicit representations, that can be utilized for efficient simulation and numerical studies of these processes, are obtained.
\end{abstract}


\section{Introduction}
The problem of defining a second-order stochastic process with a given spectrum and a given marginal distribution is of practical importance especially in control theory, signal processing and other engineering applications. A usual approach for obtaining realizations with the desired spectral and distributional properties is by using   linear time-invariant filters to white noise and then  nonlinear transformations to the filter's output, \cite{bib:Grigoriu_book}. The difficulty in the method arises from the interconnection between  the spectral density and the marginal distribution, see for example \cite{bib:BaxevaniP}. Changing the one changes the other, and the relation that links them together is not always simple or tractable. 

\cite{bib:Kay} solves this problem by proposing a model in which the distributional and spectral properties are decoupled.  The model extends the classical spectral representation by allowing frequencies to be random. More specifically he considers
\begin{equation}
\label{eq:sinusoidal}
X(t)=\frac{1}{\sqrt{m/2}}\sum_{i=1}^{m}\xi_{i}\cos(\lambda_i t +\phi_i),
\end{equation}
where the sequences of  independent and identically distributed (iid) random variables $(\xi_i)$, $(\lambda_i)$ and $(\phi_i)$ are mutually independent.  In  model (\ref{eq:sinusoidal}), the  phases $\phi_i$ need to be  uniformly distributed over any interval of length $2\pi$ to ensure the  stationarity of the  process.  On the other hand,  the distributions of   frequencies $\lambda_i$ and amplitudes $\xi_i$ can be chosen to control the spectral and marginal distribution of the signal $X$, respectively.

The role of random frequencies in controlling spectral densities has  been noticed already in earlier works, see \cite{bib:Priestley} and \cite{bib:Anderson}. However, these authors  considered   processes  consisting of a single harmonic and as such they were non-ergodic and hence of little use in practise. Summing a large number of harmonics  may to some extent alleviate this problem. For example,  \cite{bib:Kay} showed the  process (\ref{eq:sinusoidal})  to be  `nearly' ergodic in the autocorrelation for large $m$. Specifically he showed   the variance of the sample autocorrelation at time $t$ to converge as the observation time increases without bound, to
\begin{equation}
\label{mom}
\frac{\mathbb{E}[\xi^4]}{m}+\frac{\mathbb{E}[\xi^4]r_X(2t)}{m \mathbb{E}[\xi^2]}-\frac{r^2_X(t)}{m},
\end{equation}
where $\xi$ has the distribution of $\xi_i$'s and $r_X$ is the autocorrelation of process $X$.  
It is tempting to conclude  ergodicity in the autocorrelation by letting $m$ go to infinity. However,  this would  lead to  conceptual inconsistency since process $X$  depends on $m$ and formally speaking there is no process corresponding to this limiting case of $m=\infty$.
In fact as shown in \cite{bib:Kay}, if the number of harmonics $m$ in (\ref{eq:sinusoidal})  is fixed and the distribution of amplitudes $\xi_i$ is such that it does not depend on  $m$, the   distribution of the process $X$ must depend on $m$. Similarly choosing the distribution of the process $X$ independently of $m$, which is more natural in applications, gives harmonic amplitudes with distributions that depend clearly on  $m$. In either case the limiting behavior of the distributions depending on $m$, when $m$  increases without bound, is non-trivial and  requires special care when formal arguments are presented.  We overcome these difficulties  by proposing a certain natural extension to model (\ref{eq:sinusoidal}) that allows  treating imbedding it within the class of harmonizable processes and thus avoiding the above difficulties following from dealing with  finite sums with a specific number of harmonics.

The ergodic behavior of a process is one of the important properties  for which correct analysis under large $m$ is critical. 
In \cite{bib:Kay}, ergodicity in the autocorellation is studied. Under this assumption,  sample estimators of the mean and the covariance function are consistent and therefore  converge in probability to the estimated values. It happens that the argument goes through (\ref{mom}) only under the assumption of amplitudes  $\xi_i$ having distribution that does not depend on $m$, which is obviously not true when  the distribution of $X$ is independent of  $m$. We demonstrate that in most observed cases  the autocorrelation estimator is not consistent, with its variance  that converges to a  random limit rather than to zero (which would grant the consistency).
By  using the framework of harmonizable processes, we apply Birkoff's ergodic theorem to obtain convergence with probability one for time averages of any functional of the process.  It is important to realize that the strong consistency does not hold, as was already shown in \cite{bib:Wright} for the case of sample variance.     Specifically, it was shown   that the sample variance converges with probability one  to a random quantity that equals the sum of squares of the jumps of the sample process. Thus, it was demonstrated that  signals consisting of harmonics with random frequencies are never ergodic according to the  mathematical definition of ergodicity.

Moreover, we show that the observed ergodicity for large $m$ for the case when distribution of  the amplitudes of the harmonics not depending on $m$ can be interpreted as the convergence of the underlying process $X$ to a Gaussian process. We observe a trade-off here, either we want the process not to be gaussian but then one needs to accept non-ergodicity or if one demands ergodicity we return to the Gaussian domain.


To summarize, in this work, we propose the proper mathematical framework so that all the spectral, distributional and ergodic  properties of signals of the form in (\ref{eq:sinusoidal})  can be treated in a unified manner and with mathematical rigor. At the same time, the aforementioned  processes obtain a simple and interpretable stochastic representation that can be utilized for simulation purposes  and numerical studies. Moreover, the separation between the spectral density and the marginal distribution is preserved, with the former being controlled by the distribution of the frequencies and the latter by the distribution of the amplitudes.

\section{Signal harmonics with random frequencies}
\label{sec:prob_framework}
We provide a proper mathematical framework for representing signals made of harmonics featuring random frequencies. 
Following \cite{bib:Kay}, who extended the investigation of such signals  beyond the single harmonic case that was treated earlier in \cite{bib:Anderson,bib:Priestley}, let us consider
\begin{equation}
\label{eq:sinusoidal_sum}
X_{m}(t)=\sum_{i=1}^{m}\xi_{i,m}\cos(\lambda_it+\phi_i).
\end{equation}
All three sequences $(\xi_{i,m})$, $(\lambda_i)$, and $(\phi_i)$ are considered to be random, the first two made of positive variables while the $\phi_i$'s are uniformly distributed over $[0,2\pi)$, and $m$ is, at least for now, a fixed positive integer. 
Often one is interested in considering a large number of frequencies  $m$  and studying the convergence of time averages over  large periods  $T$, which in mathematical terms means studying asymptotic properties of the process with both $m$ and $T$ converging to infinity. 
However, the order of passing to the  limits with $m$ and $T$ should be treated with care as it may produce apparent formal inconsistencies, as pointed out in the introduction.
Thus studying the asymptotic characteristics of the signal in the  form of the series in  (\ref{eq:sinusoidal_sum}),  is not necessarily the best option since it may lead to some misinterpretation of the resulting properties. However there is an elegant and simple framework that alleviates these shortcomings by representing such signals as the so-called harmonizable processes, which have been introduced in  \cite{bib:Loeve} and their integral representation properties have been investigated in \cite{bib:Loeve, bib:Rosanov} amongst others. 
In a sense, harmonizable processes represent the limiting form of $X_m$  when the number of harmonics $m$ increases without bound. 
In mathematical literature these processes are conveniently described through L\'evy processes, see  \cite{bib:Sato, bib:Kyprianou} and references therein.   Next we give a brief review of the latter and a simple intuitive construction of a  harmonizable process in which the finite sum (\ref{eq:sinusoidal_sum}) is naturally embedded. 

The processes we consider have the form 
\begin{equation}
    \label{process_series}
    X(t)=\sum_{i=1}^{\infty}\xi_i \cos(\lambda_it+\phi_i),~t\in \mathbb R,
    \end{equation}
where $(\lambda_i)$ and $(\phi_i)$ are as before, i.e. they are two independent sequences of independent and identically distributed (i.i.d) random variables, the first defined on the half-line with some distribution function (cdf) $F_\lambda$ while the second  are uniformly distributed over $[0,2\pi)$. 
The sequence of amplitudes $(\xi_i)$ is still independent of the phases and the frequencies. 
However, deviating from (\ref{eq:sinusoidal_sum}), the non-negative amplitudes $(\xi_i)$ are no longer i.i.d but  have the form
$$
\xi_i=
\sigma_0\sqrt{\Lambda^{-1}(\Gamma_i)}R_i,
$$
where  $\sigma_0>0$ is non-random, $\Gamma_i$ are arrivals of the standard Poisson process, and  $(R_i)$ has i.i.d elements  distributed according to the standard Rayleigh distribution, i.e. according to density $re^{-r^2/2}$, and independent of everything else. Measure $\Lambda$ is defined   on $[0,\infty)$, and is the L\'evy measure   satisfying
\begin{equation}
\label{Levy}
\Lambda(\{0\})=0, \quad \int_{0}^\infty u^2~\Lambda(du) <\infty,
\end{equation}
 \cite{bib:Sato, bib:Kyprianou},
The generalized inverse $\Lambda^{-1}$  of the tail of $\Lambda$ is defined by
\begin{equation}
\label{Levy_inverse}
\Lambda^{-1}(u)=\inf\{x>0 :\Lambda\left[x,\infty\right)<u\}, u>0.
\end{equation}
For normalization purposes, that will be clear in the next section,  additionally to (\ref{Levy}), we assume
\begin{equation}
\label{Levy_restriction}
\int_{1}^\infty u~\Lambda(du)=1.
\end{equation}

We also note that for the variance mixtures of a standard Gaussian vector
$$
\mathbf W_i=(W_{1i},W_{2i})=\sqrt{
\Lambda^{-1}(\Gamma_i)}( Z_{1i},Z_{2i}),
$$ 
with $Z_{1i}$ and $Z_{2i}$  mutually independent standard normal, that are also independent of $\Gamma_i$, we have
$$
\xi_i=\sigma_0 \|\mathbf W_i\|=\sigma_0\sqrt{W_{1i}^2+W^2_{2i}}. 
$$

As we will see in the next section,   $\Lambda$ denotes the L\'evy measure of an infinitely divisible distribution with finite second moment defined  on the positive half-line. 
We shall also argue for the correctness of the infinite series in (\ref{process_series}) using its spectral representation in the form of a second order harmonizable process. 
We also point out that process $X$ is uniquely defined, up to a scale, by $(\sigma_0, F_\lambda, \Lambda)$, with the process spectrum   defined by $F_\lambda$ and $\Lambda$  responsible for the distribution of $X$.

\subsection{The finite sum with independent amplitudes}
\label{fsia}
Now, we demonstrate that the finite sums (\ref{eq:sinusoidal_sum}) embed in our model. 
Consider the case of a L\'evy measure $\Lambda$ with  support separated from zero, i.e.   for some fixed $u_0 >0$, 
\begin{equation*}
\label{eq:truncated_measure}
\Lambda[0,u_0]=0,
\end{equation*}
and let $u_*$ be the supremum over such $u_0$'s. 
It is easy to see that non-increasing $u\mapsto \Lambda[u,\infty)$ is constant and equal to $L=\Lambda[u_*,\infty)$ for $0< u\le u_*$.
Notice that $\Lambda^{-1}(\Gamma_i)=0$ for $\Gamma_i>L$ and for $\Gamma_i < L$, the generalized inverse $\Lambda^{-1}(\Gamma_i)$ is different than zero. Hence, there exists $ N : \Lambda^{-1}(\Gamma_{N})\neq 0$ and $\Lambda^{-1}(\Gamma_{N+1})=0$. 
It is easy to see that $N$ is a random variable depending on $L$ and Poisson distributed with  mean $L$, i.e. $N=N(L)$, where $N(t)$, $t>0$ is the standard Poisson process given by $\Gamma_i$, $i\in \mathbb N$. 
By truncating the series (\ref{process_series}), the infinite sum becomes a finite one with a random number of terms
\begin{equation}
\label{eq:series_Poisson}
X(t)=\sigma_0\sum_{i=1}^{N(L)}\sqrt{\Lambda^{-1}(\Gamma_i)}R_i \cos(\lambda_it+\phi_i).
\end{equation}
Conditioning on $N(L)=m$, the arrival times $\Gamma_i$, $i\le m$ of the Poisson process, are distributed as the order statistics of variables that are uniformly distributed on  $[0,L]$, yielding the following representation for the $X$ process
\begin{equation}
\label{eq:series_Poisson_2}
(X(t)\mcond N(L)=m)\stackrel{d}{=}\sum_{i=1}^{m}\xi_{i,L}\cos(\lambda_it+\phi_i),
\end{equation}
with i.i.d variables $\xi_{i,L}=\sigma_0\sqrt{\Lambda^{-1}(L{U_i})}R_i$, where ${U_i}$  uniformly distributed on $[0,1]$.
Here and in what follows $\stackrel{d}{=}$ stands for the equality in distribution. 
We observe that the above form is almost identical to (\ref{eq:sinusoidal_sum}) that was discussed in  \cite{bib:Kay}. 
In this sense, the processes discussed there are embedded in our model as the processes conditional on $N(L)=m$.
The limiting argument of $m$ increasing without bound that was used there in studying convergence of sample correlation can be now properly stated as the case when $L$, and thus $N(L)$, increases without bound. 

 More specifically, if we start with general $\Lambda$ that is not separated from zero, by considering a truncated measure
$$
\Lambda_{L}(A)=\Lambda\left(\mathbb I_{[u,\infty)}\cap A\right),
$$
where $u=\Lambda^{-1}(L)$, we obtain the approximation of $X(t)$ by $X_L(t)$ defined for $\Lambda_L$, i.e. by  the right hand side in (\ref{eq:series_Poisson}).
For $L$ increasing without bound $X_L$ clearly converges almost surely to $X$.
\subsection{Harmonizable Laplace process and Gamma L\'evy measure}
Although any   measure $\Lambda$  that satisfies  (\ref{Levy}) can be used in the presented construction, here we discuss in full detail the important case of the L\'evy measure characterizing the Gamma process. Recall that a stochastic process  $G(u)$ is called a Gamma process if it starts at zero, has independent and stationary increments, and the  increments  are distributed as   Gamma  with some scale and shape parameters. 
The corresponding L\'evy measure  $\Lambda$ is defined on $[0,\infty)$  in terms of the exponential integral function
$$
\Lambda\left([u,\infty)\right)=E_1(u)/\nu=
\frac{1}{\nu}\int_{u}^\infty \frac{e^{-x}}x ~dx,
$$
see also \cite{bib:Abramowitz}.
Thus   $\Lambda^{-1}(\Gamma_i)=E_1^{-1}(\nu \Gamma_i)$.
Since inverting the integral function can be computationally costly, it is better to utilize the shot noise series expansion of the gamma process,  given in  \cite{bib:Bondesson} and \cite{bib:Rosinski}:
\begin{equation}
\label{eq:Gamma_L}
G(u)= \nu\sum_{i=1}^{\infty}  V_i e^{-\nu \Gamma_i} \mathbb I_{[U_i,1]}(u), u\in [0,1] 
\end{equation}
where $V_i$ are i.i.d standard exponentially distributed random variables independent of both $\Gamma_i$ and $U_i$. Both sequences $\Gamma_i$ and $U_i$  are distributed as before. The increments $G(u+h)-G(u)$ have been chosen to follow the Gamma distribution with scale $\nu$ and shape $h/\nu $, so that $\mathbb{E}(G(1))=1$.

Using this shot noise expansion of   $G$, process $X(t)$ in (\ref{eq:series_Poisson}),  attains the following alternative and more explicit series representation
\begin{equation}
    \label{eq:Gamma_series2}
    X(t)\stackrel{d}{=}\sigma_0 \nu^{1/2} \sum_{i=1}^{\infty}  e^{-\nu\Gamma_i/2}V_i^{1/2} R_i\cos(\lambda_i t+\phi_i).
\end{equation} 
As it will be seen in section \ref{sec:distrib}, this  harmonizable process has  generalised Laplace distribution as marginal and for this reason will be referred to as \emph{harmonizable Laplace process}. 


Although the above series expression resembles the  representation given in (\ref{process_series}), the two series are not exactly the same since while $\sqrt{\Lambda^{-1}(\Gamma_i)}$'s in (\ref{process_series}) are non-increasing, in  (\ref{eq:Gamma_series2}) the non-increasing  $\exp(-\nu \Gamma_i)$ are multiplied by the i.i.d standard exponential variables $V_i$. 
However, the above explicit form allows for the following elegant representation when conditioned on the Poisson process corresponding to $\Gamma_i$'s. Let $N(L)$ as before denote  the random number of arrivals of the  Poisson process associated with $\Gamma_{i}$'s, where $\Gamma_{N(L)}\le L <\Gamma_{N(L)+1}$. Truncating the series to $N=N(L)$ and then conditioning on $N=m$ yields
\begin{equation}
    \label{eq:Gamma_series_trunc2}
    (X(t)\mcond N=m)\stackrel{d}{=}
    \sigma_0 \nu^{1/2} 
    \sum_{i=1}^{m}
    \left(
     \left(e^{-U_{i}}\right)^{L\nu }V_i
     \right)^{1/2} 
     R_i\cos(\lambda_i t+\phi_i)
    =
        \sum_{i=1}^{m}
 \xi_{i,L}\cos(\lambda_i t+\phi_i)
   \end{equation}
where $\xi_{i,L}= \sigma_0 \nu^{1/2}\left(
     \left(e^{-U_{i}}\right)^{L\nu }V_i
     \right)^{1/2} 
     R_i$ which yields yet another approximation of $X$ that has the form (\ref{eq:sinusoidal_sum}).  
 Similarly, as before by defining for $L>0$
 $$
 X_{L}(t){=}
    \sigma_0 \nu^{1/2} 
    \sum_{i=1}^{N(L)}
     e^{-\nu\Gamma_i/2}V_i^{1/2} R_i\cos(\lambda_i t+\phi_i),
 $$
 we obtain an almost sure approximation of $X$ by the finite sum  $X_{L}$ with Poisson number of terms. 
 We note that this approximation has different distribution of the i.i.d amplitudes than the one based on direct inversion of the L\'evy measure and thus illustrating the importance of the choice of asymptotic approach in the approximation of $X$. 
\section{Spectral Representation}
\label{spre}
In this section we  discuss how the  infinite sum of harmonics in (\ref{process_series}) can be represented as a harmonizable process 
\begin{equation}
\label{har}
X(t)=\int_{-\infty}^{\infty} e^{i\lambda t} d\zeta(\lambda),
\end{equation}
with $\zeta$ a  complex-valued stochastic measure,  symmetric about zero, i.e. $\zeta(-A)=\overline{\zeta(A)}$,  the overline indicates the complex conjugate, and independently scattered on the  real line. The terms spectral measure and spectral process for $\zeta$ will be used indistinguishably.  Harmonizable processes  have been introduced in \cite{bib:Loeve} and further studied in \cite{bib:CambanisHardinWerron, bib:Yaglom} and \cite{bib:Cramer} amongst others, as a first step generalization of weakly stationary mean-square continuous stochastic processes. While harmonizable processes  do not need to  satisfy the second-order property, we consider only this case originally treated in \cite{bib:Loeve}. So from  now on, the  stochastic measure $\zeta$ has finite second order moment.  Bochner's theorem for the covariance function of  $X$ yields  a spectral measure $F$ that controls the spectral process $\zeta$, which from  now on  will be denoted by $\zeta_F$. Without losing generality we assume that $F$ does not have an atom at zero. 

For the purposes of this paper we are only interested in symmetric around zero spectral distributions. Hence, $\zeta_F$  can be viewed as the zero concatenation of two processes defined for $\lambda>0$, that  have necessarily  the form
\begin{equation}
 \label{eq:spectral_meas}   
\begin{split}
&\zeta_F(0,\lambda]=\frac{\sqrt{2}}{2}\bigg(B(G(F(0,\lambda]))+i\tilde{B}(G(F(0,\lambda]))\bigg),\\
&\zeta_F[-\lambda,0)=\overline{\zeta_F(0,\lambda]},
\end{split} 
\end{equation}
where $B$ and $\tilde{B}$ are two independent standard Brownian motions. $G: [0,\infty)\mapsto [0, \infty)$ is a non-decreasing pure jump L\'evy process   with  $\mathbb{E}(G(1))=1$, which is exactly property (\ref{Levy_restriction}),  and independent of both $B$ and $\tilde{B}$. Finally $F: \mathbb R  \mapsto [0,F(\mathbb R)]$  is the  spectral distribution function controlling the spectral process $\zeta_F$. 

The  \emph{Inverse L\'evy Measure Method} \cite{bib:Rosinski}, yields  the following  series representation for the process $G$ with $ u\in [0,F(0,\infty)]$ :
\begin{equation}
    \label{LevyPr}
    G(u)=F[0,\infty)\sum_{i=1}^{\infty} \Lambda^{-1}(\Gamma_i)~ \mathbb I_{\displaystyle [U_i,1]}(u/ F(0,\infty)),\\ 
\end{equation}
where   $U_i$  are  i.i.d  random variables distributed as uniform  on  $[0,1]$, and $\Gamma_i$ are  the arrival times  of a standard Poisson process. Sequences ($U_i$)  and  $(\Gamma_i)$ are mutually independent.  $\Lambda$ denotes the  L\'evy measure of the infinitely divisible distribution of the random variable $Y=G(F(0,\infty))/F(0,\infty)$ that  satisfies (\ref{Levy}) and (\ref{Levy_restriction})  and its generalized inverse is defined as in (\ref{Levy_inverse}).
The explicit connection between $G$ and $\Lambda$ is through the characteristic function of $Y$ that is given in terms of the  L\'evy-Khinchine representation
$$
\phi_Y(t)=\exp\left(\int_0^\infty \big(e^{ixt}-1-it x\mathbb I_{(0,1)}(x)\big) ~ d\Lambda(x)\right).
$$
It is easy to see, using (\ref{Levy_restriction}) that $\phi_Y'(t)|_{t=0}=i$ so that $\mathbb{E}(Y)=-i\cdot \phi_Y'(0)=1$. Then, from the definition of the random variable $Y$, it is straightforward to deduce that  $\mathbb{E}(G(1))=1$. 

\begin{Proposition}
\label{prop:spectral}
The stochastic process $X(t)$ defined in (\ref{process_series}) has the spectral representation
\begin{equation}
    \label{process_real}
    X(t)\stackrel{d}{=} \int_{-\infty}^{\infty} e^{i\lambda t}~d\zeta_F(\lambda).
\end{equation}
with spectral process $\zeta_F$ defined as in (\ref{eq:spectral_meas} ) with the L\'evy process $G$ in (\ref{LevyPr})  defined so that $\Lambda$ is the L\'evy measure of $Y=G(F(0,\infty))/F(0,\infty)$, and 
$$
F(0,\lambda]=\frac{\sigma_0^2}{2} F_\lambda(\lambda), ~\lambda\ge 0.
$$ 
\end{Proposition}
\begin{proof}
Replacing $u$  in (\ref{LevyPr}) by $F(0,\lambda]$ and using the facts that $F(0,\infty)=\frac{\sigma_0^2}{2}$ and  $F_{\lambda}$ is a cumulative distribution function, we obtain
\begin{equation}
\label{eq:G_series}
G(F(0,\lambda])=F(0,\infty)\sum_{i=1}^{\infty} \Lambda^{-1}(\Gamma_i) \mathbb I_{[F_\lambda^{-1}(U_i),\infty]}(\lambda).
\end{equation}

Embedding the  series expansion of  $G$, given  in (\ref{eq:G_series}), in  the spectral process $\zeta_F$ in   (\ref{eq:spectral_meas}) and denoting by $\sigma_0^2/2$ the $F(0,\infty)$,  we obtain the following expression for the process $X$ defined in (\ref{process_real})
\begin{eqnarray}
    \label{process_series_new}
    X(t)&\stackrel{d}{=}\sigma_0\sum_{i=1}^{\infty}\sqrt{\Lambda^{-1}(\Gamma_i)}\cdot
     \left( Z_i^{(1)}\cos(F_\lambda^{-1}(U_i)t)- Z_i^{(2)}\sin(F_\lambda^{-1}(U_i)t)\right)\\ \nonumber
    &\stackrel{d}{=}\sigma_0\sum_{i=1}^{\infty}\sqrt{\Lambda^{-1}(\Gamma_i)}R_i \cos(F_\lambda^{-1}(U_i)t+\phi_i)\\ \nonumber
   & \stackrel{d}{=}\sigma_0\sum_{i=1}^{\infty}\sqrt{\Lambda^{-1}(\Gamma_i)}R_i \cos(\lambda_it+\phi_i) \nonumber
\end{eqnarray}
where $(Z_i^{(j)}), j=1,2$, are sequences of independent standard normal random variables. Moreover,  $(R_i)$ and $(\phi_i)$ are   sequences of independent standard Rayleigh and uniform on $[0,2\pi)$  random variables respectively. $R_i$ and $\phi_i$ are also jointly independent. The first equivalence in (\ref{process_series_new}) follows   by noticing  that the increments of a Brownian motion are stationary, independent and distributed as normal random variables with mean zero and variance $\sigma_0^2\Lambda^{-1}(\Gamma_i)$. The second equivalence follows from $\left(Z_i^{(1)},Z_i^{(2)}\right)\stackrel{d}{=}R_i \left(\cos \phi_i, \sin \phi_i\right)$. Finally $(\lambda_i)$ is a sequence of random variables distributed according to the distribution $F_\lambda$. 
\end{proof}
\begin{Remark}
 Notice that the series representation of $X(t)$ in (\ref{process_series_new}) decomposes the process into a series of real harmonics with phases and amplitudes that are random and independent of each other but also, departing from the classical case, frequencies that are also random and distributed according  to $F_\lambda$. Hence this is a method that can be used to produce stochastic processes with given marginal density and given spectral distribution function.
\end{Remark}

\subsection{The discrete spectrum } 
We finish this section  by having a closer look to the case of discrete  frequency distribution $F_{\lambda}$. It turns out,  following from spectral representation in (\ref{process_real}), that in this case the randomness of the frequencies  is only apparent.  More precisely, let us consider a spectrum that us concentrated at the frequencies $\pm l_k, k=1,\dots$  with masses $\nu_k=F_{\lambda}(\{l_k\}), \forall k$. We  have used   the standard convention $-l_k=l_{-k}$. Then,   process $X$ in (\ref{process_real}) reduces to the following sum of harmonics with non-random frequencies 
\begin{equation*}
    \label{process_series_discrete}
    X(t)\stackrel{d}{=}\sum_{k=1}^{\infty}\sqrt{G_k}\cdot \left( Z_k^{(1)}\cos(l_kt)- Z_k^{(2)}\sin(l_k t)\right)\stackrel{d}{=}\sum_{k=1}^{\infty}\sqrt{G_k}R_k \cos(l_kt+\phi_k),
\end{equation*}
where $G_k$ are independent random variables  with distribution uniquely identified by the  L\'evy measure $\Lambda$ of the process $G$ through $\mathbb{E} (G_k)=\nu_k$. The variables $Z_k^{(1)}, Z_k^{(2)}, R_k$ and $\phi_k$ are distributed as before.
\section{Distributional Properties}
\label{sec:distrib}
In this section we study the distributional properties of the harmonizable process  $X$ in (\ref{process_real}). We start with a short discussion of the distribution of the different components in the series representation of the process before we turn to the marginal  distribution of the process  and the joint distribution of the process and its derivative. The latter is important when one is interested in the level crossing sampling distributions of the process which are expressed by means of the Rice formula, \cite{bib:Rice1944,bib:Rice1945,bib:AbergP}.

\subsection{Phase distribution} 
In the previous sections, we have considered processes  with  the random amplitudes   and  frequencies having  arbitrary distributions, but with the  phases  being uniformly distributed over  some interval of length $2\pi$. The latter is a restriction that needs to be imposed  for  the process $X$  in (\ref{process_real}) to be strictly stationary. Next, we show  correctness of this statement for the case of a  single harmonic. Extension of the  argument for any number of harmonics is rather straightforward. 

Consider the single harmonic
\begin{equation}
\label{eq:single_harmonic}
\xi\cos(\lambda t+\phi)
\end{equation}
and  condition on $\lambda=\lambda_0$ and $\xi=\xi_0$. Then  for each $n\in \mathbb N$, $a_i, t_i\in \mathbb R$, $i=1,\dots, n$, $s\in \mathbb R$ we have  
\begin{align*}
\sum_{i=1}^n a_i \xi_0 \cos(\lambda_0(t_i+s)+\phi)&=\sum_{i=1}^n a_i \xi_0 \cos(\lambda_0t_i+\phi+\lambda_0s)\\
&=\sum_{i=1}^n a_i \xi_0 \cos(\lambda_0t_i+\tilde \phi)\\
&\stackrel{d}{=}\sum_{i=1}^n a_i \xi_0 \cos(\lambda_0t_i+\phi).
\end{align*}
The last equality holds if phases $\tilde \phi = (\phi+\lambda_0s)_{\mod 2 \pi}$ have the same distribution as $\phi$, which is true if and only if the latter are  uniformly distributed over $(0,2\pi]$ and thus independent of $\lambda_0$ and, of course, of $\xi_0$.
This proves conditional and thus also unconditional stationarity in the strict sense for the single harmonic, independently of the distribution of $\lambda$ and $\xi$. 
\begin{figure}[!t]
 \centering
 \includegraphics[width=3.5in]{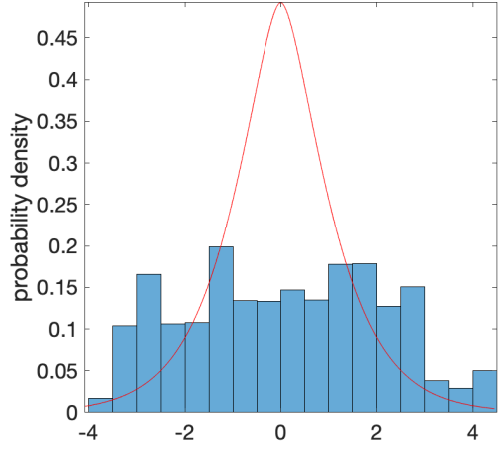}
 \caption {Theoretical density of $X(0)$ defined in (\ref{eq:Gamma_series2}) and histogram of $X(t)$ from one realisation based on 2000 simulated values  for $\nu=1.5$ which corresponds to a process with standard Laplace marginal distribution.  The frequencies follow a Uniform distribution on $[0,1]$ and the rest of the random variables are as indicated in the text.}
\label{fig:hist1}
\end{figure}

\begin{figure}[!t]
 \centering
 \includegraphics[width=3.5in]{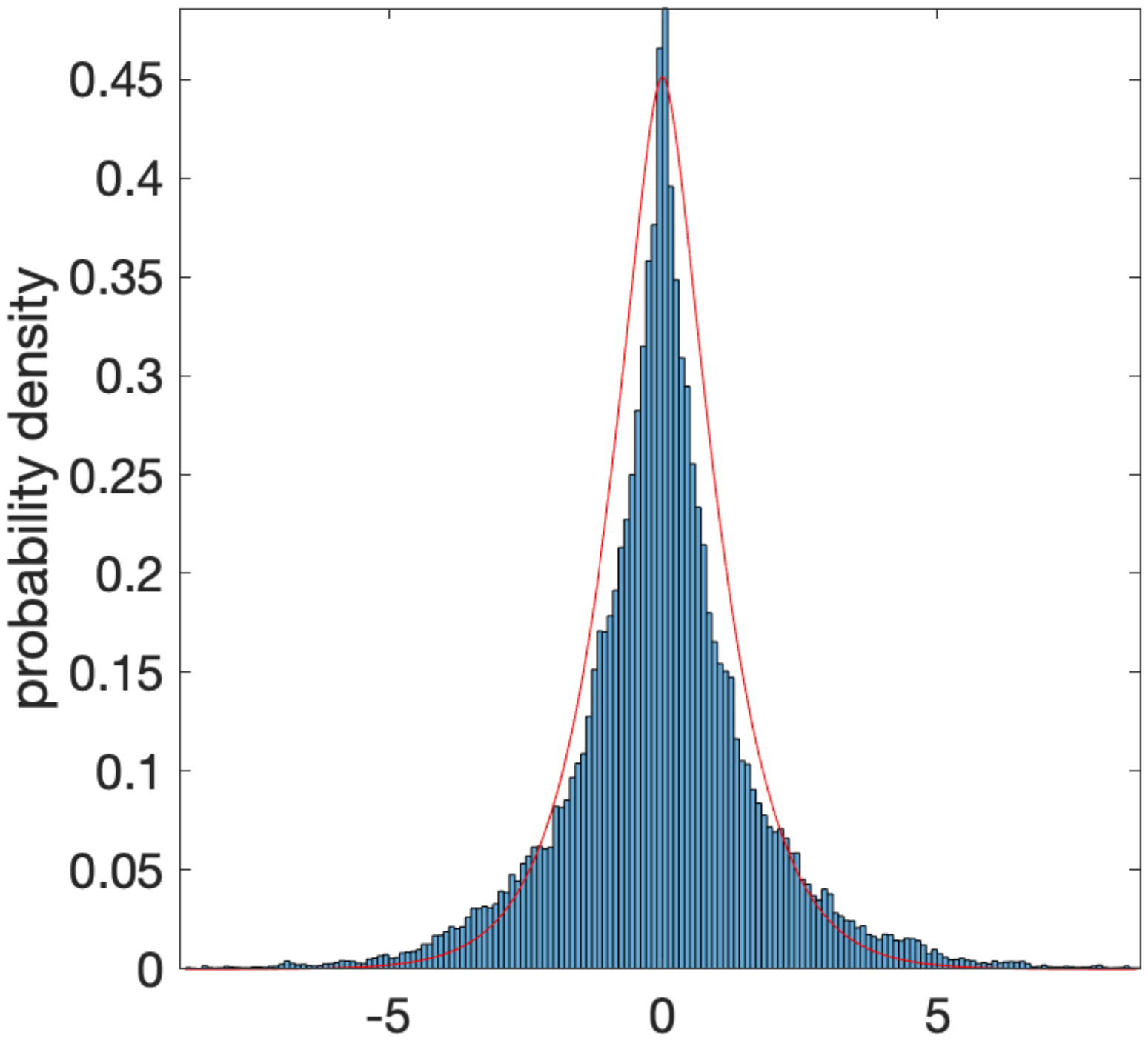}
 \caption {Theoretical density of $X(0)$ defined in (\ref{eq:Gamma_series2}) and histogram of $X(t)$  50 realisations of 2000 simulated values  for $\nu=1.5$ which corresponds to a process with standard Laplace marginal distribution.  The frequencies follow a Uniform distribution on $[0,1]$ and the rest of the random variables are as indicated in the text.}
\label{fig:hist2}
\end{figure}
For the record, we formulate this in the form of a theorem. 
\begin{Theorem}
A process of the form (\ref{eq:single_harmonic}) with  arbitrary amplitude and frequency distribution  is strictly stationary if and only if $\phi$ is uniform on $[0,2\pi)$. 
\end{Theorem}

\subsection{Marginal distribution}
The next result lists basic facts about the distribution of the process $X$.
\begin{Proposition}
\label{prop:distribution}
The harmonizable process $X(t)$ defined in (\ref{process_real}) has
\begin{enumerate}
    \item one dimensional marginal distribution of the $G$-type given  by 
    \begin{equation}
        \label{eq:1dim}
        X(t)\stackrel{d}{=}\sqrt{2\cdot G(F(0,\infty))}\cdot Z,
    \end{equation}
    where $Z$ is a standard normal variable while $G$ and $F$  are as in Proposition~\ref{prop:spectral},
    \item finite dimensional distributions defined by the characteristic function
\begin{eqnarray}
\label{eq:fdd}
&\mathbb{E}\bigg(\exp\big(i\sum_{j=1}^nu_jX(t_j)\big)\bigg)=\exp\bigg( \int_{0}^{\infty} \int_{0}^{\infty} 
\bigg(\exp\big(- x\frac{\mathbf{u}\mathbf A\mathbf{u}^T}{2}\big)-\\
&-1-x\frac{\mathbf{u}\mathbf A\mathbf{u}^T}{2}\mathbb I_{(0,1)}(x)\bigg)~d\Lambda(x)~dF(\lambda)\bigg) \nonumber
\end{eqnarray}
where $\mathbf u=(u_1,\dots,u_n)$, the transpose of $u$ is denoted by $u^T$  and the entries of the matrix $\mathbf A$ are $\cos(\lambda(t_j-t_k))$,
\item  characteristic function of $(X(0),X'(0))$ is given by
\begin{eqnarray*}
 \label{eq:dim_der}
 &\phi(u_1,u_2)= \int_0^\infty \int_0^\infty \exp\bigg(-x(u_1^2+u_2^2\lambda)/2-1-x(u_1^2+u_2^2\lambda)/2 \cdot \mathbb I_{(0,1)}(x)\bigg)~d\Lambda(x)~dF(\lambda),
\end{eqnarray*}
where the derivative of $X(t)$, which is assumed to exist in the mean-square sense,  is given by 
\begin{equation*}
X'(t)=-\int_{0}^\infty \lambda \sin(\lambda t)~dB(G(F(\lambda)))+\int_{0}^\infty \lambda \cos(\lambda t)~d\tilde{B}(G(F(\lambda))).
\end{equation*}

\end{enumerate}
\end{Proposition}
\begin{proof}

1) Relation (\ref{eq:1dim}) is immediate from noticing that since process $X(t)$ is stationary, it has the same marginal distribution as the random variable $X(0)$. To obtain the distribution of the latter notice that using the definition of the spectral measure in (\ref{eq:spectral_meas}), this is the same as the distribution of $\sqrt{2}B(G(F(0,\infty))$, which is  normal  with variance $2G(F(0,\infty))$.

2) Consider the random variable $T=\sum_{j=1}^nu_jX(t_j)$. Then 
$$
T=\int_{0}^{\infty} f(\lambda)~dB(G(F(\lambda)))+ \int_{0}^{\infty} \tilde{f}(\lambda)~d\tilde{B}(G(F(\lambda))),
$$
with 
$$
f(\lambda)=\sum_{j=1}^n u_j\cos(\lambda t_j), ~ \tilde{f}(\lambda)=\sum_{j=1}^n u_j\sin(\lambda t_j).
$$

Conditionally on $G$ we have
$$
T\mcond G\stackrel{d}{=}\int_{0}^{\infty} f(\lambda)~dB_{G\circ F}(\lambda)+ \int_{0}^{\infty} \tilde{f}(\lambda)~d\tilde{B}_{G\circ F}(\lambda).
$$
Since this  is the sum of two independent Gaussian random variables we get
\begin{equation}
    \label{eq:char_T_cond}
\mathbb{E} \big(e^{iuT}\mcond G\big)=\exp\bigg(-\frac{u^2}{2}\int_0^{\infty}\big(f^2(\lambda)+\tilde{f}^2(\lambda)\big)~dG_{F}(\lambda)\bigg),
\end{equation}
which we recognize as the Laplace transform evaluated at $u^2/2$ of the random variable  $V=\int_0^{\infty}\big(f^2(\lambda)+\tilde{f}^2(\lambda)\big)~d
G_{F}(\lambda)$.

Using  the well known form of characteristic function for a L\'evy random variable $G(1)$\cite{bib:Sato} with L\'evy measure $\Lambda$,
$$
\mathbb{E}(e^{iuG(1)})=\exp\bigg( \int_0^{\infty} \big(e^{i ux}-1-iu x \mathbb I_{(0,1)(x)}\big)~d\Lambda(x)\bigg)
$$

together with (\ref{eq:char_T_cond}), we obtain for the random variable $T$
\begin{eqnarray*}
\begin{aligned}
&\mathbb{E}\big(e^{iT}\big)= \mathbb{E}\bigg(\mathbb{E}\big(e^{iT}\mcond G\big)\bigg)=\\
&=\exp\bigg(\int_0^{\infty} \int_0^\infty \bigg(e^{-\frac{1}{2}(f^2(\lambda)+\tilde{f}^2(\lambda))x}- 1-\frac{1}{2}(f^2(\lambda)+\tilde{f}^2(\lambda))  x\mathbb I_{(0,1)}(x)\bigg)~d\Lambda(x)~dF(\lambda)\bigg).
\end{aligned}
\end{eqnarray*}
Now the proposition follows by observing that
$$
f^2(\lambda)+\tilde{f}^2(\lambda)=\sum_{j=1}^n\sum_{k=1}^nu_ju_k\cos(\lambda(t_j-t_k))=\mathbf u \mathbf A\mathbf u^T.
$$

3) The result is an immediate consequence of the second part if we  notice that 
\begin{equation*}
u_1X(0)+u_2X'(0)= \int_0^\infty u_1~dB(G(F(\lambda)))+\int_0^\infty u_2\lambda ~d\tilde{B}(G(F(\lambda))),
\end{equation*}
which is the random variable $T$ for $f(\lambda)=u_1$ and $\tilde{f}(\lambda)=u_2\lambda$. 
\end{proof}

\subsection{The harmonizable Laplace process} We conclude this section with the  special case of the harmonizable Laplace process given in (10).   In this case,  $L_F(\lambda):=B(G(F(\lambda))$ and $\tilde{L}_F(\lambda):=\tilde{B}(G(F(\lambda))$ with $G$ a Gamma process, are two Laplace processes. The distribution of $L$ and $\tilde{L}$  is an immediate consequence of the representation of the Laplace random variable as a Gaussian random variable subordinated to a gamma variable,    see \cite{bib:KotzKP}.  The resulting representation of the process in (10) is 
\begin{equation}
\label{eq:Laplace}
X(t)\stackrel{d}{=}\int_0^\infty\cos(\lambda t)~dL_F(\lambda)+\int_0^\infty\sin(\lambda t)~d\tilde{L}_F(\lambda).
\end{equation}

The marginal distributions of the process can be obtained in a similar way as in Proposition ~\ref{prop:distribution}, and are gathered in the next result.
\begin{Proposition}
\label{prop:Laplace}
The Laplace harmonizable process defined in (\ref{eq:Laplace}) has 
\begin{enumerate}
    \item marginal distribution defined by the characteristic function
    \begin{equation}
        \label{eq:laplace_marg}
        \phi_{X(t)}(u)=\bigg(1+ \nu u^2\bigg)^{- F(0,\infty)/\nu},
    \end{equation}
    \item finite dimensional distributions defined by the following characteristic function
    \begin{equation}
        \label{eq:laplace_fdd}
        \mathbb{E}\exp\big(i\sum_{j=1}^n u_jX(t_j)\big)=\exp \bigg(-\int_0^\infty \ln\big(1+ \nu\mathbf{u}\mathbf A\mathbf u^T \big)~dF(\lambda)\bigg),
    \end{equation}
    with $\mathbf u$ and $\mathbf A$ as in Proposition~\ref{prop:distribution}.
\end{enumerate}
\end{Proposition}
\begin{proof}
1. Relation (\ref{eq:laplace_marg}) follows directly from (\ref{eq:1dim}) if we consider for $G$ a gamma process with scale $\nu$, shape $1/\nu$,  and use the expression for the Laplace transform  for a Gamma random variable.

2. The form of the characteristic function in (\ref{eq:laplace_fdd}) can be derived either from (\ref{eq:fdd}) for Gamma L\'evy measure $\Lambda$ or can be found in \cite{bib:AbergP}.
\end{proof}

\section{The time averages and ergodic properties}
The ergodicity of a signal is an important property as it allows to retrieve statistical characteristics of the model by applying time averages on a single realization. 
Probably the most standard model exhibiting ergodicity is a moving average that represent a filter of a noise.
This model can feature both arbitrary spectrum and desirable distributional properties, see \cite{bib:AbergP,bib:BaxevaniP}. 
On the other hand, if a model lacks ergodicity it may provide a framework for studying phenomena for which sample-to-sample stochastic variability is significant. 

In \cite{bib:Kay}, it was suggested that the sum of $m$ harmonics with random frequencies and with the amplitudes distributed according to a random variable $\xi$ have the  ergodicity in autocovariance property and the key argument required that $m$ is increasing without bound and $\mathbb{E}(\xi^4)/m$ converges to zero, 
see the middle of the first column at page 3450 and the bottom of the first column at page 3457 in the quoted work. 
If the amplitude distribution of $\xi$ does not depend on $m$ this would reduce the requirement to the existence of a finite fourth moment for $\xi$. 
However, typically the  distribution of $\xi$ does depend on $m$ or otherwise the distribution of $X(t)$ must depend on $m$ which is difficult to interpret when we assume that $m$ is changing. 
Existence of the limiting distribution of such process has to be examined and it may lead to undesirable distribution as seen next. 

We use our leading example of Laplace (double exponential) distribution to show that, the considered  amplitude has distribution of the form
$$
\xi\stackrel{d}{=}K \sqrt{m} R \sqrt{G_m},
$$
where $K>0$ is a numerical constant independent of $m$, $R$ is distributed as Rayleigh, and $G_m$ as gamma  with  shape parameter $a/m$, for some $a>0$ and scale 1/2. 
For this distribution, the fourth moment is proportional to $m^2\cdot a/m(a/m+1)=a^2+am$, and thus, clearly, $\mathbb{E}(\xi^4)/m$ converges to $a$, which is non-zero. Thus the argument for ergodicity in covariance fails. 

On the other hand if we keep the amplitude distribution to be independent of $m$, by considering for example
$$
\xi\stackrel{d}{=}K  R \sqrt{G_1},
$$
the limiting behavior of the distribution of the process $X_m$ becomes Gaussian. 
In order to formulate this in mathematical terms, the passage to the limit with $m$ should be formulate in terms of $m=N(L)$, where $N$ is a Poisson process.
Recall that conditionally on $N(L)=m$ we have 
\begin{equation}
\sigma_0\sum_{i=1}^{N(L)}\sqrt{\Lambda^{-1}(\Gamma_i)}R_i \cos(\lambda_it+\phi_i)=\sigma_0\sum_{i=1}^{N(L)}\sqrt{\Lambda^{-1}(L U_i)}R_i \cos(\lambda_it+\phi_i),
\end{equation}
where $U_i$ is uniform on $[0,1]$. 
Thus, the case of the amplitude distribution not depending on $m$ requires that  $\Lambda^{-1}(L U_i)$ is not depending on $L$, which can be obtained by replacing $\Lambda$ by 
$$\Lambda_L(\cdot)=L\cdot \Lambda(\cdot)
$$
and  consider instead the process
\begin{equation}
\label{eq:X_L}
X_L(t)=\sigma_0\sum_{i=1}^{\infty}\sqrt{\Lambda_L^{-1}(U_i)}R_i \cos(\lambda_it+\phi_i),
\end{equation}
with all random variables defined the same way as in the proof of Proposition~\ref{prop:spectral}.
Gaussianity of the limiting behavior of $X_L$ is then formulated as follows. 
\begin{Proposition}
When $L$ increases without bound, the process $X_L/\sqrt{L}$ in (\ref{eq:X_L}) converges in distribution to the Gaussian process  $X(t)$ having spectral representation
$$
X(t)=\int_{-\infty}^{\infty} e^{i\lambda t} d\zeta_F(\lambda),
$$
with $\zeta_F(0,\lambda]=\frac{\sqrt{2}}{2}\bigg(B(F(0,\lambda])+i\tilde{B}(F(0,\lambda]))\bigg)$, $\zeta_F[-\lambda,0)=\overline{\zeta_F(0,\lambda]}$ and $B, \tilde{B}$ and $F$ as in section~\ref{spre}.
\end{Proposition}
\begin{proof}
It is quite clear that by taking the L\'evy measure $\Lambda_L=L\Lambda$, process $X_L/\sqrt{L}$ has the spectral representation of Section~\ref{spre} with the pure jump process
$$
G_L(v)=G(L\cdot v)/L,
$$
where $G$ is the corresponding jump process in the spectral representation of a process with the L\'evy measure $\Lambda$. 
Consequently, in the spectral measure of $X_L/\sqrt{L}$ given by
\begin{equation}
\zeta_{F,L}(0,\lambda]=\frac{\sqrt{2}}{2}\bigg(B(G_L(F(0,\lambda]))+i\tilde{B}(G_L(F(0,\lambda]))\bigg)
\end{equation}
the subordinator has the form
$$
G_L(F(0,\lambda])=\frac{G(L\cdot F(0,\lambda])}{L\cdot F(0,\lambda]} F(0,\lambda].
$$
For L\'evy subordinators  with finite moment, we have
$$
\lim_{L\rightarrow \infty} \frac{G(L\cdot F(0,\lambda])}{L\cdot F(0,\lambda]} =1.
$$
From this it follows easily that $\zeta_{F,L}$ converges to $\zeta_F$.
The conclusion then can be easily obtained for example, by conditioning on $G$ and establishing second order convergence of two Gaussian processes. 
\end{proof}
\begin{Remark}
From the above result it follows that in order to argue about the ergodicity for large $L$ ($m$) while requiring that the amplitude distribution does not depend on $L$ ($m$), one inevitably ends up with the process that asymptotically is Gaussian. 
\end{Remark}

In what follows, we show that the ergodicity property does not hold for processes of the form (\ref{eq:series_Poisson_2}) or equivalently (\ref{eq:Gamma_series_trunc2}). 
Despite this lack of ergodicity, the model can be still valuable for practical applications since we also establish the limiting distributions of the time averages expressed in terms of the amplitude distributions. 

We will use  \emph{ergodicity in autocorrelation}, which is a variant of ergodicity of a process $X(t)$ requiring that for a centered process 
 \begin{equation*}
     \label{eq:Birkoff}
 \lim_{T\mapsto \infty}\frac{1}{T}\int_0^T X(t)X(t+\tau)~dt=r_X(\tau).
 \end{equation*}

\begin{Proposition}
\label{prop:time_averages}
Consider  the conditioned process $(X(t)\mcond N(L)=m)$  in (\ref{eq:series_Poisson_2}) defined as finite sum of harmonics, i.e. 
$$
\tilde{X}(t)\stackrel{d}{=}\sum_{i=1}^m \xi_{i,L}\cos(\lambda_i t+\phi_i)
$$
with amplitudes, frequencies and phases being random and described earlier is Subsection~\ref{fsia}. 
Then
    \begin{align}
    \label{eq:sum_1}
    \lim_{T\mapsto \infty}\frac{1}{T}\int_0^{\infty} \tilde{X}(t)~dt&=0,\\
   \label{eq:sum_2}
       \lim_{T\mapsto \infty}\frac{1}{T}\int_0^{\infty} \tilde{X}(t)\tilde{X}(t+\tau)~dt&=\sum_{i=1}^m \frac{\xi_{i,L}^2}{2}\cos(\lambda_i\tau).
\end{align}
 \end{Proposition}
 \begin{proof}
The relation in (\ref{eq:sum_1}) is a direct consequence of the properties of integrals of trigonometric functions and the fact that 
 $$
 \lim_{T\mapsto \infty} \frac{\sin(T)}{T}=0
 $$
To prove (\ref{eq:sum_2}) notice that for  $\tau>0$,
 \begin{eqnarray*}
 \label{eq:proof_2}
  &\frac{1}{T}\int_0^{T} \tilde{X}(t)\tilde{X}(t+\tau)~dt=\\
  &\frac{1}{T}\sum_{i,j=1}^m \xi_{i,L}\xi_{j,L} \int_{0}^{T}\cos(\lambda_it+\phi_i)\cos(\lambda_{j}(t+\tau)+\phi_j)~dt=\\  
  &\frac{1}{T}\sum_{i,j=1}^m \frac{\xi_{i,L}\xi_{j,L}}{2} \bigg(\int_{0}^{T}\cos\big((\lambda_{i}+\lambda_{j})t+\lambda_j\tau +\phi_i+\phi_j\big)~dt + \\
  &+\int_{0}^{T}\cos\big((\lambda_{i}-\lambda_{j})t -\lambda_j\tau +\phi_i-\phi_j\big)~dt\bigg)
 \end{eqnarray*}
 Notice that the first integral in the last equation  is of the form
 $$
 \int_0^T \cos(at+b)~dt=\frac{1}{a}(\sin(a'T+b')-\sin(a'')),
 $$
 for some constants $a,b, a',b'$ and $a''$, so that divided by $T$ as $T$ increases without bound goes to zero. The same is true for the second integral also, unless $i=j$, giving
 $$
\lim_{T\mapsto \infty}\frac{1}{T}\sum_{i=1}^m \frac{\xi_{i,L}^2}{2}  \int_0^T \cos(\lambda_i\tau)~dt = \sum_{i=1}^m\frac{\xi_{i,L}^2}{2}\cos(\lambda_i\tau).
 $$

 \end{proof}

  Concluding we present  the main result.
  \begin{Theorem}
  The  process ${X}$ defined  in (\ref{process_series}) that in a harmonizable form is also given in Proposition~\ref{prop:spectral} is not ergodic in autocorrelation.
  Moreover, almost surely we have
  \begin{align*}
    \lim_{T\mapsto \infty}\frac{1}{T}\int_0^{\infty} {X}(t)~dt&=0,\\
       \lim_{T\mapsto \infty}\frac{1}{T}\int_0^{\infty} {X}(t){X}(t+\tau)~dt&=\sum_{i=1}^\infty 
       \frac{
       \Lambda^{-1}(\Gamma_i) R_i^2
       }{2}\cos(\lambda_i\tau).
\end{align*}
  \end{Theorem}
  \begin{proof}
  This result follows from Proposition~\ref{prop:time_averages}  and convergence of $X_L$  introduced at the end of Subsection~\ref{fsia} to $X$.
  Standard arguments for the change of the order of the two limits, one in $L$ and the other in $T$ are omitted here. 
 \end{proof}
 
\section{Conclusions}
The important model of stationary stochastic signals arising from allowing for random frequencies in harmonics has been thoroughly studied by representing it as a harmonizable process with respect to non-Gaussian measure. It features  distributional properties decoupled from spectral properties, as the former are governed by the distribution of  amplitudes and the latter by the distribution of frequencies. One important property that is obtained from that approach is that the time averages do preserve randomness over the long time horizon. This can be viewed as important property in some applications. On the other hand it follows that if the ergodic property is required from the model for practical reasons, one has to resort to other models that feature arbitrary spectrum and non-Gaussian marginal distributions, such as non-Gaussian moving averages considered in \cite{bib:PodgorskiW}, \cite{bib:BaxevaniPW}. These models   also feature non-Gaussian tail distributions but they require some coupling (dependence) between different harmonics both in amplitudes and frequencies, see \cite{bib:BaxevaniP}. 
Despite that the asymptotics of the time averages cannot retrieve such properties of the model as the moments, autocovariance, and sample distributions, the ergodic theorem yields the stochastic form of the limiting values that allows for the model estimation and studies of its sample to sample properties.



%
\bibliographystyle{apalike}
\bibliography{RefHL}

\end{document}